\documentclass{amsart}

\newtheorem{theorem}[equation]{Theorem}
\newtheorem{lemma}[equation]{Lemma}
\newtheorem{corollary}[equation]{Corollary}
\newtheorem{proposition}[equation]{Proposition}
\newtheorem{conjecture}[equation]{Conjecture}

\newtheorem{example}{Example}

\numberwithin{equation}{section}

\usepackage{amscd,amsmath,amssymb,verbatim}

\begin{document}

\title[On logarithmic solutions of $A$-hypergeometric systems]{On logarithmic solutions of \\ $A$-hypergeometric systems} 
\author{Alan Adolphson}
\address{Department of Mathematics\\
Oklahoma State University\\
Stillwater, Oklahoma 74078}
\email{adolphs@math.okstate.edu}
\author{Steven Sperber}
\address{School of Mathematics\\
University of Minnesota\\
Minneapolis, Minnesota 55455}
\email{sperber@math.umn.edu}
\date{\today}
\begin{abstract}
For an $A$-hypergeometric system with parameter $\beta$, a vector $v$ with minimal negative support satisfying $Av=\beta$ gives rise to a logarithm-free series solution.  We find conditions on $v$ analogous to `minimal negative support' that guarantee the existence of logarithmic solutions of the system and we give explicit formulas for those solutions.  Although we do not study in general the question of when these logarithmic solutions lie in a Nilsson ring, we do examine the $A$-hypergeometric systems corresponding to the Picard-Fuchs equations of certain families of complete intersections and we state a conjecture regarding the integrality of the associated mirror maps.
\end{abstract}
\maketitle

\section{Introduction}

Let $A=\{{\bf a}_1,\dots,{\bf a}_N\}\subseteq{\mathbb Z}^n$ and let $L\subseteq{\mathbb Z}^N$ be the lattice of relations on $A$:
\[ L = \bigg\{ l=(l_1,\dots,l_N)\in{\mathbb Z}^N\:\bigg|\: \sum_{i=1}^N l_i{\bf a}_i = {\bf 0} \bigg\}. \]
Let $\beta = (\beta_1,\dots,\beta_n)\in{\mathbb C}^n$.  The {\it $A$-hypergeometric system with parameter $\beta$\/} is the system of partial differential operators in $\lambda_1,\dots,\lambda_N$ consisting of the {\it box operators\/}
\begin{equation}
\Box_l = \prod_{l_i>0}\bigg(\frac{\partial}{\partial\lambda_i}\bigg)^{l_i} - \prod_{l_i<0}\bigg(\frac{\partial}{\partial \lambda_i}\bigg)^{-l_i}\quad\text{for $l\in L$}
\end{equation}
and the {\it Euler\/} or {\it homogeneity operators}
\begin{equation}
Z_i = \sum_{j=1}^N a_{ij}\lambda_j\frac{\partial}{\partial \lambda_j} -\beta_i \quad\text{for $i=1,\dots,n$,}
\end{equation}
where ${\bf a}_j = (a_{1j},\dots,a_{nj})$.  

To simplify notation we define for $z\in{\mathbb C}$ and $k\in {\mathbb Z}$, $k<-z$ if $z\in{\mathbb Z}_{<0}$,
\[ [z]_k = \begin{cases} 1 & \text{if $k=0$,} \\ \frac{1}{(z+1)(z+2)\cdots(z+k)} & \text{if $k>0$,} \\
z(z-1)\cdots(z+k+1) & \text{if $k<0$.} \end{cases} \]

For $z=(z_1,\dots,z_N)\in{\mathbb C}^N$ and $k=(k_1,\dots,k_N)\in{\mathbb Z}^N$ we define
\[ [z]_k = \prod_{i=1}^N [z_i]_{k_i}. \]
The {\it negative support\/} of $z$ is the set
\[ {\rm nsupp}(z) = \{i\in \{1,\dots,N\}\mid \text{$z_i$ is a negative integer}\}. \]

Let $v=(v_1,\dots,v_N)\in{\mathbb C}^N$ satisfy $\sum_{i=1}^N v_i{\bf a}_i = \beta$.  One says that $v$ has {\it minimal negative support\/} if there is no $l\in L$ for which ${\rm nsupp}(v+l)$ is a proper subset of ${\rm nsupp}(v)$.  Let
\[ L_v = \{l\in L\mid {\rm nsupp}(v+l) = {\rm nsupp}(v)\} \]
and let
\begin{equation}
F(\lambda) = \sum_{l\in L_v} [v]_l\lambda^{v+l}.
\end{equation}
By \cite[Proposition 3.4.13]{SST} (see also Section 3 below), the series $F(\lambda)$ is a solution of the $A$-hypergeometric system (1.1), (1.2) if and only if $v$ has minimal negative support.

Let $G(\lambda) = \sum_{l\in L} b_l\lambda^{v+l}$, where $b_l\in{\mathbb C}$.  Note that any such $G(\lambda)$ satisfies the Euler operators (1.2).  We call $\{l\in L\mid b_l\neq 0\}$ the {\it support\/} of $G(\lambda)$.  We say that $F(\lambda)\log\lambda_i+G(\lambda)$ is a {\it quasisolution\/} if is satisfies the box operators (1.1).  We call a set $\{F(\lambda)\log\lambda_i + G_i(\lambda)\}_{i=1}^N$ a {\it complete set of quasisolutions\/} if each element is a quasisolution.

\begin{proposition}
Suppose $\{F(\lambda)\log\lambda_i + G_i(\lambda)\}_{i=1}^N$ is a complete set of quasisolutions and let $l=(l_1,\dots,l_N)\in L$.  Then
\begin{equation}
\sum_{i=1}^N l_i\big(F(\lambda)\log\lambda_i + G_i(\lambda)\big) = F(\lambda)\log\lambda^l + \sum_{i=1}^N l_iG_i(\lambda)
\end{equation}
is a solution of the $A$-hypergeometric system $(1.1)$, $(1.2)$.
\end{proposition}

\begin{proof}
The left-hand side of (1.5) is a linear combination of solutions of the box operators, hence also satisfies the box operators.  The monomials in $F(\lambda)$ and the $G_i(\lambda)$ satisfy the Euler operators.  The fact that the right-hand side of~(1.5) also satisfies the Euler operators then follows from the following elementary fact: if a monomial $\lambda^c$ satisfies the Euler operators and $l\in L$, then $\lambda^c\log\lambda^l$ also satisfies the Euler operators.
\end{proof}

The purpose of this article is to give conditions on the vector $v$ that guarantee the existence of a complete set of quasisolutions.  For $v=(v_1,\dots,v_N)\in{\mathbb C}^N$ and $i\in\{1,\dots,N\}$, define the {\it $\hat{\imath}$-negative support of $v$\/} to be
\[ \text{$\hat{\imath}$-${\rm nsupp}(v)$} = \{j\in\{1,\dots,\hat{\imath},\dots,N\} \mid \text{$v_j$ is a negative integer}\}, \]
where the symbol `$\hat{\imath}$' indicates that the element $i$ has been deleted from the set $\{1,\dots,N\}$.  We say that $v$ has {\it minimal $\hat{\imath}$-negative support\/} if $\hat{\imath}$-${\rm nsupp}(v+l)$ is not a proper subset of $\hat{\imath}$-${\rm nsupp}(v)$ for any $l\in L$.  Let
\[ L_{v,\hat\imath} = \{ l\in L\mid \text{$\hat{\imath}$-${\rm nsupp}(v+l)=\hat{\imath}$-${\rm nsupp}(v)$}\}. \]
Note that $L_v\subseteq L_{v,\hat{\imath}}$.  Our main result is the following theorem.

\begin{theorem}
Suppose that $v$ has minimal negative support and minimal $\hat{\imath}$-negative support.  Then there exists a quasisolution $F(\lambda)\log\lambda_i + G(\lambda)$ with the support of~$G(\lambda)$ contained in $L_{v,\hat{\imath}}$.
\end{theorem}

Using Proposition 1.4 we then have the following corollary.

\begin{corollary}
If $v$ has minimal negative support and minimal $\hat{\imath}$-negative support for all $i\in\{1,\dots,N\}$, then for each $l\in L$ there exists a solution $F(\lambda)\log\lambda^l + G_l(\lambda)$ of the $A$-hypergeometric system $(1.1)$, $(1.2)$ with the support of $G_l(\lambda)$ contained in~$\bigcup_{i=1}^N L_{v,\hat{\imath}}$.
\end{corollary}

In Section 4 we prove Theorem 1.6 by giving an explicit formula for the quasisolution (see Theorem 4.11 below).  Similar results hold for solutions involving higher powers of logarithms.  We treat the quadratic case in Section 5, which provides an outline for the general case.  We do not consider in general the question of when the logarithmic solutions lie in a Nilsson ring; however, in Section 6 we examine the case of $A$-hypergeometric systems corresponding to the Picard-Fuchs equations of certain families of complete intersections.

\section{Solving the box operators}

We present a method for generating solutions of the box operators (1.1).  In Section 3, we explain how to recover the logarithm-free solution $F(\lambda)$ using this method and in Section 4 we explain how to use it to construct quasisolutions.

Our scheme is based on finding sequences $f_z^{(k)}(t)$ of functions of one variable $t$, parametrized by $z\in{\mathbb C}$ and indexed by $k\in{\mathbb Z}$, satisfying for all $k$
\begin{equation}
\frac{d}{dt}\big(f_z^{(k)}(t)\big) = f_z^{(k-1)}(t).
\end{equation}
To do this, simply choose $f_z^{(0)}(t)$ and successively differentiate it to find the $f_z^{(k)}(t)$ for $k<0$ and successively integrate it to find the $f_z^{(k)}(t)$ for $k>0$.

For the applications in this article, we shall take
\begin{equation}
f_z^{(0)}(t) = t^z\log^m t\quad\text{($m$ a nonnegative integer).}
\end{equation}
We first describe the $f_z^{(k)}(t)$ for $k<0$.  For a positive integer $i$, let $S_{i,j}(x_1,\dots,x_i)$ be the $j$-th elementary function in $i$ variables:
\[ S_{i,0} = 1,\quad S_{i,1} = x_1+\cdots+x_i,\quad \dots, \quad S_{i,i} = x_1\cdots x_i. \]
Set $s_{i,j}(z) = S_{i,j}(z,z-1,\dots,z-i+1)$.  Differentiation of (2.2) gives
\begin{multline}
f_z^{(k)}(t) = \\
t^{z+k}\sum_{i=0}^{\min\{-k,m\}} s_{-k,-k-i}(z)m(m-1)\cdots(m-i+1)\log^{m-i}t\quad\text{for $k<0$.}
\end{multline}
The functions $f_z^{(k)}(t)$ for $k>0$ depend on whether or not $z$ is a negative integer.  For a positive integer $k$, let $M_{k,i}(x_1,\dots,x_k)$ be the sum of all monomials of degree~$i$ in $k$ variables:
\[ M_{k,i}(x_1,\dots,x_k) = \sum_{j_1+\cdots+j_k = i} x_1^{j_1}\cdots x_k^{j_k}. \]
Set $m_{k,i}(z) = M_{k,i}\big((z+1)^{-1},\dots,(z+k)^{-1}\big)$.  If $z\not\in{\mathbb Z}_{<0}$, then (always choosing the constant of integration to be $0$ so that $f_z^{(k)}(t)$ is homogeneous of degree $z+k$ in $t$)
\begin{multline}
f_z^{(k)}(t) = \\
[z]_kt^{z+k}\sum_{i=0}^m (-1)^i m(m-1)\cdots(m-i+1)m_{k,i}(z)\log^{m-i}t\quad\text{for $k>0$.}
\end{multline}
If $z\in{\mathbb Z}_{<0}$, then (2.4) is valid for $0<k<-z$, while for $k\geq -z$ we have
\begin{equation}
f_z^{(k)}(t) = t^{z+k}\cdot\text{(polynomial of degree $m+1$ in $\log t$).}
\end{equation}
In this last case we do not need a precise formula as our hypotheses will guarantee that these $f_z^{(k)}(t)$ do not appear in our formulas for solutions of $A$-hypergeometric systems.

Let $v=(v_1,\dots,v_N)\in{\mathbb C}^N$ satisfy $\sum_{i=1}^N v_i{\bf a}_i = \beta$ and for each $i=1,\dots,N$ choose a family of functions $\{f_{v_i}^{(k)}(t)\}_{i\in{\mathbb Z}}$ satisfying (2.1) and replace $t$ by $\lambda_i$.  The functions $f_{v_i}^{(k)}(\lambda_i)$ then satisfy
\begin{equation}
\frac{\partial}{\partial\lambda_j}\big( f_{v_i}^{(k)}(\lambda_i)\big) = \begin{cases} f_{v_i}^{(k-1)}(\lambda_i) & \text{if $i=j$,} \\ 0 & \text{if $i\neq j$} \end{cases}
\end{equation}
for $i,j=1,\dots,N$ and all $k\in{\mathbb Z}$.  Let $T_1,\dots,T_n$ be indeterminates.  If ${\bf b}=(b_1,\dots,b_n)\in{\mathbb C}^n$, we write $T^{\bf b} = T_1^{b_1}\cdots T_n^{b_n}$.  We form the generating series
\begin{equation}
\Phi_{v_i}(\lambda_i,T) = \sum_{k\in{\mathbb Z}} f_{v_i}^{(k)}(\lambda_i)T^{(v_i+k){\bf a}_i}.
\end{equation}
Equation (2.6) is equivalent to
\begin{equation}
\frac{\partial}{\partial\lambda_j}\big(\Phi_{v_i}(\lambda_i,T)\big) = \begin{cases} T^{{\bf a}_i}\Phi_{v_i}(\lambda_i,T) & \text{if $i=j$,} \\ 0 & \text{if $i\neq j$.} \end{cases}
\end{equation}

Let $\Phi_v(\lambda,T)$ be the product of these generating series:
\begin{equation}
\Phi_v(\lambda,T) = \prod_{i=1}^N \Phi_{v_i}(\lambda_i,T) = \sum_{u\in{\mathbb Z}A} \Phi_{v,u}(\lambda) T^{\beta+u},
\end{equation}
where 
\begin{equation}
\Phi_{v,u}(\lambda) = \sum_{\sum k_i{\bf a}_i = u} \prod_{i=1}^N f_{v_i}^{(k_i)}(\lambda_i).
\end{equation}
The sum on the right-hand side of (2.10) may be infinite, in which case we are implicitly assuming that this infinite sum is a well-defined element of some module over the ring of differential operators in the $\lambda_i$ with polynomial coefficients.

The main point of this section is the following result.
\begin{proposition}
For all $u\in{\mathbb Z}A$ and all $l\in L$, one has $\Box_l\big(\Phi_{v,u}(\lambda)\big) = 0$.
\end{proposition}

\begin{proof}
Let $l=(l_1,\dots,l_N)\in L$.  We have
\[ \Box_l\big(\Phi_v(\lambda,T)\big) = \sum_{u\in{\mathbb Z}A} \Box_l\big(\Phi_{v,u}(\lambda)\big)T^{\beta+u}, \]
so to prove the proposition it suffices to show that $\Box_l(\Phi_v(\lambda,T)) = 0$.  It follows from Equations (2.8) and (2.9) that
\begin{equation}
\frac{\partial}{\partial\lambda_j}\big(\Phi_v(\lambda,T)\big) = T^{{\bf a}_j} \Phi_v(\lambda,T)
\end{equation}
for all $j$, therefore
\[ \Box_l\big(\Phi_v(\lambda,T)\big) = \big( T^{\sum_{l_i>0} l_i{\bf a}_i} - T^{-\sum_{l_i<0}l_i{\bf a}_i}\big) \Phi_v(\lambda,T) = 0 \]
since $\sum_{l_i>0}l_i{\bf a}_i = -\sum_{l_i<0}l_i{\bf a}_i$.  
\end{proof}

\section{Logarithm-free solutions}

We focus attention on the case $u={\bf 0}$ of Proposition 2.11, i.e., we examine the solution $\Phi_{v,{\bf 0}}(\lambda)$ of the box operators (1.1) to determine when it also satisfies the Euler operators (1.2).  By (2.10) we have
\begin{equation}
\Phi_{v,{\bf 0}}(\lambda) = \sum_{l\in L} \prod_{i=1}^N f_{v_i}^{(l_i)}(\lambda_i).
\end{equation}

To obtain logarithm-free solutions, we take $m=0$ in Equations (2.2)--(2.5).  This gives
\begin{equation}
f_z^{(k)}(t) = [z]_kt^{z+k}
\end{equation}
if $z\in{\mathbb C}\setminus{\mathbb Z}_{<0}$ or if $z\in{\mathbb Z}_{<0}$ and $k<-z$.  It gives
\begin{equation}
f_z^{(k)}(t) = t^{z+k}\cdot\text{(linear polynomial in $\log t$)}
\end{equation}
if $z\in{\mathbb Z}_{<0}$ and $k\geq -z$.  

Suppose first that $v=(v_1,\dots,v_N)\in({\mathbb C}\setminus{\mathbb Z}_{<0})^N$.  Then (3.2) gives $f_{v_i}^{(k)}(\lambda_i) = [v_i]_k\lambda_i^{v_i+k}$ for $i=1,\dots,N$ and (3.1) becomes
\begin{equation}
\Phi_{v,{\bf 0}}(\lambda) = \sum_{l\in L} [v]_l\lambda^{v+l}.
\end{equation}
Since $\sum_{i=1}^N v_i{\bf a}_i = \beta$, a trivial computation shows that the formal series (3.4) satisfies the Euler operators (1.2), hence (3.4) is a formal solution of the $A$-hypergeomet\-ric system.

More generally, if $v\in{\mathbb C}^N$ has some negative integer coordinates, then some terms given by (3.3) may appear in (3.1) and $\Phi_{v,{\bf 0}}(\lambda)$ will not satisfy any Euler equation because of the presence of the $\log\lambda_i$-terms.  However, if no terms from (3.3) appear, then it will satisfy an $A$-hypergeometric system.  Note that from (3.2) and (3.3),
\begin{equation}
f_{v_i}^{(l_i)}(\lambda_i) = 0 \quad\text{if and only if $v_i\in{\mathbb Z}_{\geq 0}$ and $l_i<-v_i$,}
\end{equation}
hence $\prod_{i=1}^N f_{v_i}^{(l_i)}(\lambda_i)=0$ unless ${\rm nsupp}(v+l)\subseteq{\rm nsupp}(v)$.  We therefore always have
\begin{equation}
\Phi_{v,{\bf 0}}(\lambda) = \sum_{l\in L'_v} \prod_{i=1}^N f_{v_i}^{(l_i)}(\lambda_i),
\end{equation}
where $L'_v = \{l\in L\mid {\rm nsupp}(v+l)\subseteq{\rm nsupp}(v)\}$.  From (3.2) and (3.3) we observe that
\begin{equation}
\text{$\log\lambda_i$ appears in $f_{v_i}^{(l_i)}(\lambda_i)$ if and only if $v_i\in{\mathbb Z}_{<0}$ and $l_i\geq -v_i$.}
\end{equation}
For $l\in L'_v$, this happens if and only if ${\rm nsupp}(v+l)$ is a proper subset of ${\rm nsupp}(v)$.  So if we assume that $v$ has minimal negative support, then $L'_v = L_v$ and $f_{v_i}^{(l_i)}(\lambda_i)$ is given by (3.2) for all $l\in L_v$ and all $i$.  We therefore have
\begin{equation}
\Phi_{v,{\bf 0}}(\lambda) = \sum_{l\in L_v}[v]_l\lambda^{v+l}.
\end{equation}
It is clear that all monomials on the right-hand side of (3.8) satisfy Equations~(1.2).  It follows that $\Phi_{v,{\bf 0}}(\lambda)$ is a solution of the $A$-hypergeometric system (1.1), (1.2) if and only if $v$ has minimal negative support.  Note that the solution $\Phi_{v,{\bf 0}}(\lambda)$ given by (3.8) is the solution denoted $F(\lambda)$ in (1.3).

\section{Logarithmic solutions}

To obtain the logarithmic solutions described in Section 1, we shall also need the case $m=1$ of Equations (2.2)--(2.5).  Taking $m=1$ gives
\begin{equation}
f_z^{(0)}(t) = t^z\log t,
\end{equation}
\begin{equation}
f_z^{(k)}(t) = [z]_kt^{z+k}\log t + s_{-k,-k-1}(z)t^{z+k}\quad\text{for $k<0$,}
\end{equation}
and, if $z\in{\mathbb C}\setminus{\mathbb Z}_{<0}$ or if $z\in{\mathbb Z}_{<0}$ and $k<-z$,
\begin{equation}
f_z^{(k)}(t) = [z]_kt^{z+k}\big(\log t-m_{k,1}(z)\big)\quad\text{for $k>0$.}
\end{equation}
If $z\in{\mathbb Z}_{<0}$ and $k\geq -z$, then
\begin{equation}
f_z^{(k)}(t) = t^{z+k}\cdot\text{(polynomial of degree 2 in $\log t$).}
\end{equation}

Let $v\in{\mathbb C}^N$.  Fix $i\in\{1,\dots,N\}$ and let $f_{v_i}^{(k)}(t)$ be defined by (4.1)--(4.4).  For $j\in\{1,\dots,N\}$, $j\neq i$, let $f_{v_j}^{(k)}(t)$ be defined by (3.2)--(3.3).  We have the associated generating series
\[ \Phi_{v_m}(\lambda_m,T) = \sum_{k_m\in{\mathbb Z}}f_{v_m}^{(k_m)}(\lambda_m)T^{(v_m+k_m){\bf a}_m} \]
for $m\in\{1,\dots,N\}$ and their product
\begin{equation}
\Phi_{v,i}(\lambda,T) = \prod_{m=1}^N \Phi_{v_m}(\lambda_m,T) = \sum_{u\in{\mathbb Z}A} \Phi_{v,i,u}(\lambda)T^{\beta+u},
\end{equation}
where
\begin{equation}
\Phi_{v,i,u}(\lambda) = \sum_{\sum k_m{\bf a}_m = u} \prod_{m=1}^N f_{v_m}^{(k_m)}(\lambda_m).
\end{equation}
By Proposition 2.11, the $\Phi_{v,i,u}(\lambda)$ satisfy the box operators (1.1).

We focus attention on the case $u={\bf 0}$:
\begin{equation}
\Phi_{v,i,{\bf 0}}(\lambda) = \sum_{l\in L}\prod_{m=1}^N f_{v_m}^{(l_m)}(\lambda_m).
\end{equation}
Let
\[ L'_{v,\hat{\imath}} = \{l\in L\mid \hat{\imath}\text{-}{\rm nsupp}(v+l)\subseteq\hat{\imath}\text{-}{\rm nsupp}(v)\}. \]
If $l\in L\setminus L'_{v,\hat{\imath}}$, then there exists $j\neq i$ such that $v_j+l_j\in{\mathbb Z}_{<0}$ and $v_j\in{\mathbb Z}_{\geq 0}$.  By~(3.5) we then have $f_{v_j}^{(l_j)}(\lambda_j) = 0$, so (4.7) becomes
\begin{equation}
\Phi_{v,i,{\bf 0}}(\lambda) = \sum_{l\in L'_{v,\hat{\imath}}}\prod_{m=1}^N f_{v_m}^{(l_m)}(\lambda_m).
\end{equation}

We now suppose that $v$ has minimal $\hat{\imath}$-negative support.  This implies that $L'_{v,\hat{\imath}} = L_{v,\hat{\imath}}$, so the sum over $L'_{v,\hat{\imath}}$ in (4.8) becomes a sum over $L_{v,\hat{\imath}}$.  Furthermore, if $l\in L_{v,\hat{\imath}}$, there does not exist $j\neq i$ such that $v_j\in{\mathbb Z}_{<0}$ and $v_j+l_j\in{\mathbb Z}_{\geq 0}$.  In particular, no $f_{v_j}^{(l_j)}(\lambda_j)$ with $j\neq i$ is given by (3.3), all are given by (3.2).  We can thus rewrite~(4.8) as
\begin{equation}
\Phi_{v,i,{\bf 0}}(\lambda) = \sum_{l\in L_{v,\hat{\imath}}} f_{v_i}^{(l_i)}(\lambda_i)\prod_{\substack{j=1 \\ j\neq i}}^N [v_j]_{l_j}\lambda_j^{v_j+l_j}.
\end{equation}

Now assume also that $v$ has minimal negative support.  For $l\in L_{v,\hat{\imath}}$, this implies that if $v_i\in{\mathbb Z}_{<0}$, then $v_i+l_i\in{\mathbb Z}_{<0}$ also (otherwise, ${\rm nsupp}(v+l)$ would be a proper subset of ${\rm nsupp}(v)$).  It follows that for $l\in L_{v,\hat{\imath}}$, no $f_{v_i}^{(l_i)}(\lambda_i)$ is given by (4.4), all are given by (4.1)--(4.3).  This allows us to write (4.9) as
\begin{multline}
\Phi_{v,i,{\bf 0}}(\lambda) = \\
\sum_{l\in L_{v,\hat{\imath}}} [v]_l\lambda^{v+l}\log\lambda_i + \sum_{l\in L_{v,\hat{\imath}}}\lambda^{v+l}\prod_{\substack{j=1\\ j\neq i}}^N [v_j]_{l_j}\cdot \begin{cases} 0 & \text{if $l_i=0$,} \\ -[v_i]_{l_i}m_{l_i,1}(v_i) & \text{if $l_i>0$,} \\ s_{-l_i,-l_i-1}(v_i) & \text{if $l_i<0$.} \end{cases}
\end{multline}

Finally, suppose that $l\in L_{v,\hat{\imath}}$ but $l\not\in L_v$.  Since $v$ has minimal negative support, this implies that $v_i\in{\mathbb Z}_{\geq 0}$ but $v_i+l_i\in{\mathbb Z}_{<0}$, hence $[v_i]_{l_i}=0$.  We can thus replace the first sum on the right-hand side of (4.10) by a sum over $L_v$.  We summarize this discussion with the following explicit version of Theorem 1.6.
\begin{theorem}
Suppose that $v$ has minimal $\hat{\imath}$-negative support and minimal negative support.  Then $\Phi_{v,i,{\bf 0}}(\lambda) = F(\lambda)\log\lambda_i + G_i(\lambda)$ is a quasisolution, where
\[ F(\lambda) = \sum_{l\in L_v} [v]_l\lambda^{v+l} \]
and
\[ G_i(\lambda) =  \sum_{l\in L_{v,\hat{\imath}}}\lambda^{v+l}\prod_{\substack{j=1\\ j\neq i}}^N [v_j]_{l_j}\cdot \begin{cases} 0 & \text{if $l_i=0$,} \\ -[v_i]_{l_i}m_{l_i,1}(v_i) & \text{if $l_i>0$,} \\ s_{-l_i,-l_i-1}(v_i) & \text{if $l_i<0$.} \end{cases} \]
\end{theorem}

\begin{example}{\rm 
Let ${\bf a}_1,{\bf a}_2,{\bf a}_3,{\bf a}_4\in{\mathbb R}^3$ be the columns of the matrix
\[ A=\left( \begin{array}{rrrr} 1 & 0 & 0 & 1 \\ 0 & 1 & 0 & 1 \\ 0 & 0 & 1 & -1 \end{array}\right). \]
Let $a,b\in{\mathbb C}\setminus{\mathbb Z}_{\leq 0}$, $\beta = (-a,-b,0)$, and $v=(-a,-b,0,0)$.  One calculates that $L=\{(-l,-l,l,l)\mid l\in{\mathbb Z}\}$, $v$ has minimal negative support and minimal $\hat{\imath}$-negative support for all $i$, and that
\[ L_v = L_{v,\hat{1}} = \cdots =L_{v,\hat{4}} = \{(-l,-l,l,l)\mid l\in{\mathbb Z}_{\geq 0}\}. \]
We have the logarithm-free solution
\begin{equation}
F(\lambda) = \lambda_1^{-a}\lambda_2^{-b}\sum_{l=0}^\infty [-a]_{-l}[-b]_{-l}([0]_l)^2 \bigg(\frac{\lambda_3\lambda_4}{\lambda_1\lambda_2}\bigg)^l.
\end{equation}
To find a logarithmic solution, we compute the $G_i(\lambda)$ using Theorem 4.11.  We have
\[ G_1(\lambda) = \lambda_1^{-a}\lambda_2^{-b}\sum_{l=1}^\infty [-a]_{-l}[-b]_{-l}([0]_l)^2\frac{s_{l,l-1}(-a)}{[-a]_{-l}}\bigg(\frac{\lambda_3\lambda_4}{\lambda_1\lambda_2}\bigg)^l, \]
\[ G_2(\lambda) = \lambda_1^{-a}\lambda_2^{-b}\sum_{l=1}^\infty [-a]_{-l}[-b]_{-l}([0]_l)^2\frac{s_{l,l-1}(-b)}{[-b]_{-l}}\bigg(\frac{\lambda_3\lambda_4}{\lambda_1\lambda_2}\bigg)^l, \]
and for $i=3,4$
\[ G_i(\lambda) = -\lambda_1^{-a}\lambda_2^{-b}\sum_{l=1}^\infty [-a]_{-l}[-b]_{-l}([0]_l)^2m_{l,1}(0)\bigg(\frac{\lambda_3\lambda_4}{\lambda_1\lambda_2}\bigg)^l. \]
This gives us a complete set of quasisolutions, so we may apply Proposition 1.4 using $(-1,-1,1,1)\in L$ to get the logarithmic solution
\begin{multline}
\lambda_1^{-a}\lambda_2^{-b}\bigg[\bigg(\sum_{l=0}^\infty [-a]_{-l}[-b]_{-l}([0]_l)^2 \bigg(\frac{\lambda_3\lambda_4}{\lambda_1\lambda_2}\bigg)^l\bigg)\log\bigg(\frac{\lambda_3\lambda_4}{\lambda_1 \lambda_2}\bigg) \\
-\sum_{l=1}^\infty \bigg(2m_{l,1}(0) + \frac{s_{l,l-1}(-a)}{[-a]_{-l}} + \frac{s_{l,l-1}(-b)}{[-b]_{-l}}\bigg)
[-a]_{-l}[-b]_{-l}([0]_l)^2 \bigg(\frac{\lambda_3\lambda_4}{\lambda_1\lambda_2}\bigg)^l\bigg].
\end{multline}

The formulas in this example become more recognizable when expressed in the classical Pochhammer notation: for $z\in{\mathbb C}$ and $l\geq 0$,
\[ (z)_l = z(z+1)\cdots(z+l-1). \]
One has $[-a]_{-l} = (-1)^l(a)_l$, $[-b]_{-l} = (-1)^l(b)_l$, and $[0]_l = 1/(1)_l = 1/l!$ for $l\geq 0$, so (4.12) becomes
\[ F(\lambda) = \lambda_1^{-a}\lambda_2^{-b} \sum_{l=0}^\infty \frac{(a)_l(b)_l}{l!^2}\bigg(\frac{\lambda_3 \lambda_4}{\lambda_1\lambda_2}\bigg)^l = \lambda_1^{-a}\lambda_2^{-b}{}_2F_1\bigg(a,b;1;\frac{\lambda_3 \lambda_4}{\lambda_1\lambda_2}\bigg) \]
and (4.13) becomes (using the definitions of $m_{l,1}(0)$, $s_{l,l-1}(-a)$, and $s_{l,l-1}(-b)$)
\begin{multline*}
\lambda_1^{-a}\lambda_2^{-b}\bigg[{}_2F_1\bigg(a,b;1;\frac{\lambda_3 \lambda_4}{\lambda_1\lambda_2}\bigg) 
\log\bigg(\frac{\lambda_3\lambda_4}{\lambda_1\lambda_2}\bigg) \\
 + \sum_{l=1}^\infty \frac{(a)_l(b)_l}{l!^2} \bigg(\sum_{m=0}^{l-1} \frac{1}{a+m} + \frac{1}{b+m} - \frac{2}{1+m}\bigg)\bigg(\frac{\lambda_3\lambda_4}{\lambda_1\lambda_2}\bigg)^l\bigg].
\end{multline*}
After the specialization $\lambda_i\mapsto 1$ for $i=1,2,3$ and $\lambda_4\mapsto t$, these expressions give two independent solutions at $t=0$ of the Gaussian hypergeometric equation
\[ t(1-t)y'' + (1-(a+b+1)t)y'-aby = 0. \] }
\end{example}

\begin{example} {\rm 
(see \cite[Example 3.5.2]{SST}) Let ${\bf a}_1,\dots,{\bf a}_5\in{\mathbb R}^3$ be the columns of the matrix
\[ A = \left(\begin{array}{rrrrr} 1 & 1 & 1 & 1 & 1 \\ -1 & 1 & 1 & -1 & 0 \\ -1 & -1 & 1 & 1 & 0 \end{array}\right). \]
Take $\beta = (1,0,0)$ and $v=(0,0,0,0,1)$.  One calculates that 
\[ L = \{(a,b,a,b,-2a-2b)\mid a,b\in{\mathbb Z}\}, \]
$v$ has minimal negative support and minimal $\hat{\imath}$-negative support for all $i$, and that
\[ L_v = L_{v,\hat{1}} = \cdots = L_{v,\hat{4}} = \{(0,0,0,0,0)\}, \]
while
\[ L_{v,\hat{5}} = \{(a,b,a,b,-2a-2b)\mid a,b\in{\mathbb Z}_{\geq 0}\}. \]
We get from Theorem 4.11 that $F(\lambda) = \lambda_5$, $G_i(\lambda) = 0$ for $i=1,\dots,4$, and
\begin{equation}
G_5(\lambda) = \lambda_5\sum_{\substack{a,b=0\\ (a,b)\neq(0,0)}}^\infty \frac{(2a+2b-2)!}{a!^2b!^2} \bigg(
\frac{\lambda_1\lambda_3}{\lambda_5^2}\bigg)^a\bigg(\frac{\lambda_2\lambda_4}{\lambda_5^2}\bigg)^b.
\end{equation}
Associated to $(-1,0,-1,0,2)\in L$ we have by Proposition 1.4 the solution
\[ \lambda_5\log(\lambda_5^2/\lambda_1\lambda_3) + 2G_5(\lambda); \]
associated to $(0,1,0,1,-2)\in L$ we have by Proposition 1.4 the solution
\[ \lambda_5\log(\lambda_2\lambda_4/\lambda_5^2) - 2G_5(\lambda). \] }
\end{example}

\section{Higher logarithmic solutions}

In this section we describe solutions that are quadratic in the $\log\lambda_i$.  It will then be clear how to extend this construction to obtain solutions involving higher powers of the $\log\lambda_i$.  

Let $i,j\in\{1,\dots,N\}$ (we allow $i=j$), let $F(\lambda)$, $G_i(\lambda)$, and $G_j(\lambda)$ be given by Theorem~4.11, and let $H_{ij}(\lambda)$ be a series of the form $\sum_{l\in L} b_l\lambda^{v+l}$.  We call
\begin{equation}
F(\lambda)\log\lambda_i\log\lambda_j + G_i(\lambda)\log\lambda_j + G_j(\lambda)\log\lambda_i + H_{ij}(\lambda)
\end{equation}
a {\it second-order quasisolution\/} if it satisfies the box operators (1.1).

Let $\{H_{ij}(\lambda)\}_{i,j=1}^N$ be a collection of series of the form $\sum_{l\in L} b_l\lambda^{v+l}$ satisfying $H_{ij}(\lambda) = H_{ji}(\lambda)$ for all $i,j$.  We call the collection
\begin{equation}
\{F(\lambda)\log\lambda_i\log\lambda_j + G_i(\lambda)\log\lambda_j + G_j(\lambda)\log\lambda_i + H_{ij}(\lambda) \}_{i,j=1}^N
\end{equation}
a {\it complete set of second-order quasisolutions\/} if each element is a quasisolution.  (Note that the pairs $(i,j)$ and $(j,i)$ give the same second-order quasisolution.)

\begin{proposition}
Suppose that $(5.2)$ is  a complete set of second-order quasisolutions and let $l=(l_1,\dots,l_N)$ and $l' = (l'_1,\dots,l'_N)$ be elements of $L$.  Then the expression
\begin{multline}
\sum_{i,j=1}^N l_il'_j\bigg(F(\lambda)\log\lambda_i\log\lambda_j + G_i(\lambda)\log\lambda_j + G_j(\lambda)\log\lambda_i + H_{ij}(\lambda)\bigg) = \\
F(\lambda)\log\lambda^l\log\lambda^{l'} + \bigg(\sum_{i=1}^N l_iG_i(\lambda)\bigg)\log\lambda^{l'} + \bigg(\sum_{j=1}^N l'_jG_j(\lambda)\bigg) \log\lambda^l + \sum_{i,j=1}^N l_il'_j H_{ij}(\lambda)
\end{multline}
is a solution of the $A$-hypergeometric system $(1.1)$, $(1.2)$.
\end{proposition}

\begin{proof}
The equality (5.4) is straightforward to check.  The left-hand side of (5.4) satisfies the box operators because it is a linear combination of solutions of the box operators.  If $\lambda^c$ satisfies the Euler operators and $l,l'\in L$, then $\lambda^c\log\lambda^l\log\lambda^{l'}$, $\lambda^c\log\lambda^l$, and $\lambda^c\log\lambda^{l'}$ are also solutions of the Euler operators, so the right-hand side of (5.4) satisfies the Euler operators.
\end{proof}

Let $i,j\in\{1,\dots,N\}$, $i\neq j$, and let $v=(v_1,\dots,v_N)\in{\mathbb C}^N$.  Define the 
{\it $\hat{\imath}\hat{\jmath}$-negative support of $v$\/} to be
\[ \hat{\imath}\hat{\jmath}\text{-}{\rm nsupp}(v) = \{k\in\{1,\dots,\hat{\imath},\dots,\hat{\jmath},\dots,N\}\mid \text{$v_k$ is a negative integer}\}. \]
We say that $v$ has {\it minimal $\hat{\imath}\hat{\jmath}$-negative support\/} if $\hat{\imath}\hat{\jmath}\text{-}{\rm nsupp}(v+l)$ is not a proper subset of $\hat{\imath}\hat{\jmath}\text{-}{\rm nsupp}(v)$ for any $l\in L$.  Define
\[ L_{v,\hat{\imath}\hat{\jmath}} = \{l\in L\mid \hat{\imath}\hat{\jmath}\text{-}{\rm nsupp}(v+l) = \hat{\imath}\hat{\jmath}\text{-}{\rm nsupp}(v)\}. \]
Note that $L_{v,\hat{\imath}}$ and $L_{v,\hat{\jmath}}$ are both contained in $L_{v,\hat{\imath}\hat{\jmath}}$.  

\begin{theorem}
Let $F(\lambda)$ and the $G_i(\lambda)$ be as in Theorem $4.11$. \\
{\bf (a)} If $v$ has minimal negative support and minimal $\hat{\imath}$-negative support for some $i\in\{1,\dots,N\}$, then there exists a second-order quasisolution
\[ F(\lambda)\log^2\lambda_i + 2G_i(\lambda)\log\lambda_i + H_{ii}(\lambda) \]
with the support of $H_{ii}(\lambda)$ contained in $L_{v,\hat{\imath}}$. \\
{\bf (b)} If $v$ has minimal negative support, minimal $\hat{\imath}$-negative support, minimal $\hat{\jmath}$-negative support and minimal $\hat{\imath}\hat{\jmath}$-negative support for some $i,j\in\{1,\dots,N\}$, $i\neq j$, then there exists a second-order quasisolution
\[ F(\lambda)\log\lambda_i\log\lambda_j + G_i(\lambda)\log\lambda_j + G_j(\lambda)\log\lambda_i + H_{ij}(\lambda) \]
with the support of $H_{ij}(\lambda)$ contained in $L_{v,\hat{\imath}\hat{\jmath}}$.
\end{theorem}

\noindent{\bf Remark:}  The hypothesis of Theorem 5.5(b) is somewhat redundant: if $v$ has minimal $\hat{\imath}$-negative support and minimal $\hat{\jmath}$-negative support for $i\neq j$, then $v$ has minimal negative support.

We prove Theorem 5.5 by giving explicit formulas for the second-order quasisolutions.  To prove part (a) of the theorem we take $m=2$ in Eqns.\ (2.2)--(2.5).  This gives
\begin{equation}
f_z^{(0)}(t) = t^z\log^2 t,
\end{equation}
\begin{equation}
f_z^{(k)}(t) = t^{z+k}\big([z]_k\log^2 t + 2s_{-k,-k-1}(z)\log t + 2s_{-k,-k-2}(z)\big) \quad\text{for $k<0$,}
\end{equation}
and, if $z\not\in{\mathbb Z}_{<0}$ or if $z\in{\mathbb Z}_{<0}$ and $k<-z$, then
\begin{equation}
f_z^{(k)}(t) = [z]_kt^{z+k}\big(\log^2 t - 2m_{k,1}(z)\log t + 2m_{k,2}(z)\big)\quad\text{for $k>0$.}
\end{equation}
If $z\in{\mathbb Z}_{<0}$ and $k\geq -z$, then
\begin{equation}
f_z^{(k)}(t) = t^{z+k}\cdot\text{(polynomial of degree $3$ in $\log t$).}
\end{equation}

We now proceed as in Section 4.  Let $v\in{\mathbb C}^N$ and fix $i\in\{1,\dots,N\}$.  Let~$f_{v_i}^{(k)}(t)$ be defined by (5.6)--(5.9) and for $j\in\{1,\dots,N\}$, $j\neq i$, let $f_{v_j}^{(k)}(t)$ be defined by (3.2)--(3.3).  For all $m\in\{1,\dots,N\}$ we form the generating series
\[ \Phi_{v_m}(\lambda_m,T) = \sum_{k_m\in{\mathbb Z}}f_{v_m}^{(k_m)}(\lambda_m)T^{(v_m+k_m){\bf a}_m} \]
and take their product
\begin{equation}
\Phi_{v,i}(\lambda,T) = \prod_{m=1}^N \Phi_{v_m}(\lambda_m,T) = \sum_{u\in{\mathbb Z}A} \Phi_{v,i,u}(\lambda)T^{\beta+u},
\end{equation}
where
\begin{equation}
\Phi_{v,i,u}(\lambda) = \sum_{\sum k_m{\bf a}_m=u} \prod_{m=1}^N f_{v_m}^{(k_m)}(\lambda_m).
\end{equation}
Proposition 2.11 implies that the $\Phi_{v,i,u}(\lambda)$ satisfy the box operators (1.1).

We again focus on the case $u={\bf 0}$:
\begin{equation}
\Phi_{v,i,{\bf 0}}(\lambda) = \sum_{l\in L}\prod_{m=1}^N f_{v_m}^{(l_m)}(\lambda_m).
\end{equation}
The same argument that showed (4.7) implies (4.9) gives in this case
\begin{equation}
\Phi_{v,i,{\bf 0}}(\lambda) = \sum_{l\in L_{v,\hat{\imath}}} f_{v_i}^{(l_i)}(\lambda_i) \prod_{\substack{j=1\\ j\neq i}}^N [v_j]_{l_j}\lambda_j^{v_j+l_j}.
\end{equation}
Since $v$ has minimal negative support, if $l\in L_{v,\hat{\imath}}$ and $v_i\in{\mathbb Z}_{<0}$, then $v_i+l_i\in{\mathbb Z}_{<0}$.  It follows that for $l\in L_{v,\hat{\imath}}$, no $f_{v_i}^{(l_i)}(\lambda_i)$ is given by (5.9), all are given by (5.6)--(5.8).  We can thus rewrite (5.13) as
\begin{multline}
\Phi_{v,i,{\bf 0}}(\lambda) = \sum_{l\in L_{v,\hat{\imath}}}[v]_l\lambda^{v+l}\log^2\lambda_i \\
 + \sum_{l\in L_{v,\hat{\imath}}} \lambda^{v+l}\log\lambda_i\prod_{\substack{j=1\\ j\neq i}}^N [v_j]_{l_j}\cdot 
\begin{cases} 0 & \text{if $l_i=0$,} \\ -2[v_i]_{l_i}m_{l_i,1}(v_i) & \text{if $l_i>0$,} \\ 2s_{-l_i,-l_i-1}(v_i) & \text{if $l_i<0$,} \end{cases} \\
 + \sum_{l\in L_{v,\hat{\imath}}} \lambda^{v+l}\prod_{\substack{j=1\\ j\neq i}}^N [v_j]_{l_j}\cdot 
\begin{cases} 0 & \text{if $l_i=0,-1$,} \\ 2[v_i]_{l_i}m_{l_i,2}(v_i) & \text{if $l_i>0$,} \\ 2s_{-l_i,-l_i-2}(v_i) & \text{if $l_i\leq -2$.} \end{cases}
\end{multline}

As noted in the proof of Theorem 4.11, one has $[v]_l=0$ for $l\in L_{v,\hat{\imath}}\setminus L_v$, so the first sum on the right-hand side of (5.14) can be replaced by a sum over $L_v$.  Equation (5.14) thus simplifies to
\begin{equation}
\Phi_{v,i,{\bf 0}}(\lambda) = F(\lambda)\log^2\lambda_i + 2G_i(\lambda)\log\lambda_i + H_{ii}(\lambda), 
\end{equation}
where $F(\lambda)$ and $G_i(\lambda)$ are given by Theorem 4.11 and where
\begin{equation}
H_{ii}(\lambda) = \sum_{l\in L_{v,\hat{\imath}}} \lambda^{v+l}\prod_{\substack{j=1\\ j\neq i}}^N [v_j]_{l_j}\cdot 
\begin{cases} 0 & \text{if $l_i=0,-1$,} \\ 2[v_i]_{l_i}m_{l_i,2}(v_i) & \text{if $l_i>0$,} \\ 2s_{-l_i,-l_i-2}(v_i) & \text{if $l_i\leq -2$.} \end{cases}
\end{equation}
This gives explicitly the second-order quasisolution of Theorem 5.5(a).

We now consider the assertion of Theorem 5.5(b).  Let $v\in{\mathbb C}^N$ and let $i,j\in\{1,\dots,N\}$, $i\neq j$.  Let $f_{v_i}^{(k)}(t)$ and $f_{v_j}^{(k)}(t)$ be defined by (4.1)--(4.4) and let $f_{v_m}^{(k)}(t)$ be defined by (3.2)--(3.3) for $m\neq i,j$.  For $m=1,\dots,N$ we define the associated generating series
\[ \Phi_{v_m}(\lambda_m,T) = \sum_{k_m\in{\mathbb Z}} f_{v_m}^{(k_m)}(\lambda_m)T^{(v_m+k_m){\bf a}_m} \]
and their product 
\begin{equation}
\Phi_{v,ij}(\lambda,T) = \prod_{m=1}^N \Phi_{v_m}(\lambda_m,T) = \sum_{u\in{\mathbb Z}A} \Phi_{v,ij,u}(\lambda)T^{\beta+u},
\end{equation}
where
\begin{equation}
\Phi_{v,ij,u}(\lambda) = \sum_{\sum k_m{\bf a}_m = u} \prod_{m=1}^N f_{v_m}^{(k_m)}(\lambda_m).
\end{equation}
Proposition 2.11 implies that the $\Phi_{v,ij,u}(\lambda)$ satisfy the box operators.

We consider the case $u={\bf 0}$:
\begin{equation}
\Phi_{v,ij,{\bf 0}}(\lambda) = \sum_{l\in L} \prod_{m=1}^N f_{v_m}^{(l_m)}(\lambda_m).
\end{equation}
Let
\[ L'_{v,\hat{\imath}\hat{\jmath}} = \{l\in L\mid \hat{\imath}\hat{\jmath}\text{-}{\rm nsupp}(v+l)\subseteq \hat{\imath}\hat{\jmath}\text{-}{\rm nsupp}(v)\}. \]
If $l\in L\setminus L'_{v,\hat{\imath}\hat{\jmath}}$, then there exists $m\neq i,j$ such that $v_m+l_m\in{\mathbb Z}_{<0}$ and $v_m\in{\mathbb Z}_{\geq 0}$.  Eq.~(3.5) then implies that $f_{v_m}^{(l_m)}(\lambda_m)=0$, so (5.19) becomes
\begin{equation}
\Phi_{v,ij,{\bf 0}}(\lambda) = \sum_{l\in L'_{v,\hat{\imath}\hat{\jmath}}} \prod_{m=1}^N f_{v_m}^{(l_m)}(\lambda_m).
\end{equation}

Our assumption that $v$ has minimal $\hat{\imath}\hat{\jmath}$-negative support implies that $L'_{v,\hat{\imath}\hat{\jmath}}=L_{v,\hat{\imath}\hat{\jmath}}$.  Furthermore, if $l\in L_{v,\hat{\imath}\hat{\jmath}}$ then there exists no $m\neq i,j$ for which $v_m\in{\mathbb Z}_{<0}$ and $v_m+l_m\in{\mathbb Z}_{\geq 0}$.  It follows that all such $f_{v_m}^{(l_m)}(\lambda_m)$ are given by (3.2), none is given by (3.3), so (5.20) becomes
\begin{equation}
\Phi_{v,ij,{\bf 0}}(\lambda) = \sum_{l\in L_{v,\hat{\imath}\hat{\jmath}}} f_{v_i}^{(l_i)}(\lambda_i)f_{v_j}^{(l_j)}(\lambda_j)\prod_{\substack{m=1\\ m\neq i,j}}^N [v_m]_{l_m}\lambda_m^{v_m+l_m}.
\end{equation}

Now let $l\in L_{v,\hat{\imath}\hat{\jmath}}$.  Since $v$ has minimal $\hat{\imath}$-negative support, if $v_j\in{\mathbb Z}_{<0}$, then $v_j+l_j\in{\mathbb Z}_{<0}$ also.  It follows that $f_{v_j}^{(l_j)}(\lambda_j)$ is always given by (4.1)--(4.3), never by (4.4).  Similarly, since $v$ has minimal $\hat{\jmath}$-negative support, it follows that $f_{v_i}^{(l_i)}(\lambda_i)$ is always given by (4.1)--(4.3), never by (4.4).  We now expand (5.21) by using (4.1)--(4.3) to express $f_{v_i}^{(l_i)}(\lambda_i)$ and $f_{v_j}^{(l_j)}(\lambda_j)$:
\begin{multline}
\Phi_{v,ij,{\bf 0}}(\lambda) = \sum_{l\in L_{v,\hat{\imath}\hat{\jmath}}} [v]_l\lambda^{v+l}\log\lambda_i\log\lambda_j \\
+\sum_{l\in L_{v,\hat{\imath}\hat{\jmath}}} \lambda^{v+l}\log\lambda_j\bigg(\prod_{\substack{m=1\\ m\neq i}}^N [v_m]_{l_m}\bigg)\cdot \begin{cases} 0 & \text{if $l_i=0$,} \\ -[v_i]_{l_i}m_{l_i,1}(v_i) & \text{if $l_i>0$,} \\
s_{-l_i,-l_i-1}(v_i) & \text{if $l_i<0$,} \end{cases} \\
+\sum_{l\in L_{v,\hat{\imath}\hat{\jmath}}} \lambda^{v+l}\log\lambda_i\bigg(\prod_{\substack{m=1\\ m\neq j}}^N [v_m]_{l_m}\bigg)\cdot \begin{cases} 0 & \text{if $l_j=0$,} \\ -[v_j]_{l_j}m_{l_j,1}(v_j) & \text{if $l_j>0$,} \\
s_{-l_j,-l_j-1}(v_j) & \text{if $l_j<0$,} \end{cases} \\
+\sum_{l\in L_{v,\hat{\imath}\hat{\jmath}}}\lambda^{v+l}\bigg(\prod_{\substack{m=1\\ m\neq i,j}}^N [v_m]_{l_m}\bigg) \cdot \left. \begin{cases} 0 & \text{if $l_i=0$,} \\ -[v_i]_{l_i}m_{l_i,1}(v_i) & \text{if $l_i>0$,} \\
s_{-l_i,-l_i-1}(v_i) & \text{if $l_i<0$,} \end{cases} \right\} \cdot\begin{cases} 0 & \text{if $l_j=0$,} \\ -[v_j]_{l_j}m_{l_j,1}(v_j) & \text{if $l_j>0$,} \\
s_{-l_j,-l_j-1}(v_j) & \text{if $l_j<0$.} \end{cases}
\end{multline}

Some of the terms on the right-hand side of (5.22) vanish.  Suppose that $l\in L_{v,\hat{\imath}\hat{\jmath}}$ but $l\not\in L_v$.  Our hypotheses imply that either $v_i+l_i\in{\mathbb Z}_{<0}$ and $v_i\in{\mathbb Z}_{\geq 0}$ or that $v_j+l_j\in{\mathbb Z}_{<0}$ and $v_j\in{\mathbb Z}_{\geq 0}$ (or both), i.e.,
\begin{equation}
\text{either $[v_i]_{l_i}=0$ or $[v_j]_{l_j}=0$ for $l\in L_{v,\hat{\imath}\hat{\jmath}}\setminus L_v$.}
\end{equation}
It follows that the first sum on the right-hand side of (5.22) can be replaced by a sum over $L_v$.  If $l\in L_{v,\hat{\imath}\hat{\jmath}}$ but $l\not\in L_{v,\hat{\jmath}}$, then our hypotheses imply that $v_i+l_i\in{\mathbb Z}_{<0}$ but $v_i\in{\mathbb Z}_{\geq 0}$, hence
\begin{equation}
\text{$[v_i]_{l_i} = 0$ for $l\in L_{v,\hat{\imath}\hat{\jmath}}\setminus L_{v,\hat{\jmath}}$.}
\end{equation}
Similarly,
\begin{equation}
\text{$[v_j]_{l_j} = 0$ for $l\in L_{v,\hat{\imath}\hat{\jmath}}\setminus L_{v,\hat{\imath}}$.}
\end{equation}
It follows that the second sum on the right-hand side of (5.22) can be replaced by a sum over $L_{v,\hat{\imath}}$ and the third sum on the right-hand side of (5.22) can be replaced by a sum over $L_{v,\hat{\jmath}}$.  We therefore get
\begin{equation}
\Phi_{v,ij,{\bf 0}}(\lambda) = F(\lambda)\log\lambda_i\log\lambda_j + G_i(\lambda)\log\lambda_j + G_j(\lambda) \log\lambda_i + H_{ij}(\lambda),
\end{equation}
where $F(\lambda)$, $G_i(\lambda)$, and $G_j(\lambda)$ are given by Theorem 4.11 and where
\begin{multline}
H_{ij}(\lambda) = \sum_{l\in L_{v,\hat{\imath}\hat{\jmath}}} \lambda^{v+l}\prod_{\substack{m=1\\ m\neq i,j}}^N [v_m]_{l_m}\\
 \cdot \left. \begin{cases} 0 & \text{if $l_i=0$,} \\ -[v_i]_{l_i}m_{l_i,1}(v_i) & \text{if $l_i>0$,} \\
s_{-l_i,-l_i-1}(v_i) & \text{if $l_i<0$,} \end{cases} \right\} \cdot\begin{cases} 0 & \text{if $l_j=0$,} \\ -[v_j]_{l_j}m_{l_j,1}(v_j) & \text{if $l_j>0$,} \\
s_{-l_j,-l_j-1}(v_j) & \text{if $l_j<0$.} \end{cases}
\end{multline}
This is the second-order quasisolution of Theorem 5.5(b).

\setcounter{example}{1}
\begin{example} {\rm 
(cont.) Clearly $v$ has minimal $\hat{\imath}\hat{\jmath}$-negative support for all $i,j$.  One checks that for $i=1,\dots,4$ we have
\[ L_{v,\hat{\imath}\hat{5}} = L_{v,\hat{5}} = \{(a,b,a,b,-2a-2b)\mid a,b\in{\mathbb Z}_{\geq 0}\}. \]
We also have
\[ L_{v,\hat{1}\hat{3}} = \{(a,b,a,b,-2a-2b)\mid -a\geq b\geq 0\} \]
and 
\[ L_{v,\hat{2}\hat{4}} = \{(a,b,a,b,-2a-2b)\mid -b\geq a\geq 0\}. \]
The remaining cases are trivial:
\[ L_{v,\hat{1}\hat{2}} = L_{v,\hat{1}\hat{4}} = L_{v,\hat{2}\hat{3}} = L_{v,\hat{3}\hat{4}} = \{(0,0,0,0,0)\}. \]
For $i=1,\dots,4$ we have the second-order quasisolutions
\[ \Phi_{v,ii,{\bf 0}}(\lambda) = \lambda_5\log^2\lambda_i \]
and for $i=5$ we have the second-order quasisolution
\[ \Phi_{v,55,{\bf 0}}(\lambda) = \lambda_5\log^2\lambda_5 + 2G_5(\lambda)\log\lambda_5 + H_{55}(\lambda), \]
where $G_5(\lambda)$ is given by (4.14) and where by (5.16)
\[ H_{55}(\lambda) = \lambda_5\sum_{\substack{a,b=0\\ (a,b)\neq (0,0)}}^\infty \frac{2\cdot(2a+2b-2)!}{a!^2b!^2} \bigg(1-\sum_{i=1}^{2a+2b-2}\frac{1}{i}\bigg)\bigg(\frac{\lambda_1\lambda_3}{\lambda_5^2}\bigg)^a \bigg(\frac{\lambda_2\lambda_4}{\lambda_5^2}\bigg)^b. \]
For $(i,j)\neq(i,5),(1,3),(2,4)$ we have 
\[ \Phi_{v,ij,{\bf 0}}(\lambda) = \lambda_5\log\lambda_i\log\lambda_j. \]
When $(i,j) = (i,5)$, we get
\[ \Phi_{v,i5,{\bf 0}}(\lambda) = \lambda_5\log\lambda_i\log\lambda_5 + G_5(\lambda)\log\lambda_i + H_{i5}(\lambda), \]
where from (5.27) we have for $i=1,3$
\[ H_{i5}(\lambda) = -\lambda_5\sum_{b=0}^\infty\sum_{a=1}^{\infty} \frac{(2a+2b-2)!}{a!^2b!^2}\bigg(1 + \frac{1}{2}+\cdots+\frac{1}{a}\bigg)\bigg(\frac{\lambda_1\lambda_3}{\lambda_5^2}\bigg)^a \bigg(\frac{\lambda_2\lambda_4}{\lambda_5^2}\bigg)^b \]
and for $i=2,4$
\[ H_{i5}(\lambda) = -\lambda_5\sum_{a=0}^\infty\sum_{b=1}^{\infty} \frac{(2a+2b-2)!}{a!^2b!^2}\bigg(1 + \frac{1}{2}+\cdots+\frac{1}{b}\bigg)\bigg(\frac{\lambda_1\lambda_3}{\lambda_5^2}\bigg)^a \bigg(\frac{\lambda_2\lambda_4}{\lambda_5^2}\bigg)^b. \]
Finally, we have
\[ \Phi_{v,13,{\bf 0}}(\lambda)  = \lambda_5\log\lambda_1\log\lambda_3 + H_{13}(\lambda), \]
where
\[ H_{13}(\lambda) = \lambda_5\sum_{b=0}^\infty \sum_{\substack{a=-b\\ a\neq 0}}^{-\infty} \frac{(-a-1)!^2}{b!^2 (-2a-2b+1)!}\bigg(\frac{\lambda_1\lambda_3}{\lambda_5^2}\bigg)^a \bigg(\frac{\lambda_2\lambda_4}{\lambda_5^2}\bigg)^b, \]
and
\[ \Phi_{v,24,{\bf 0}}(\lambda) = \lambda_5\log\lambda_2\log\lambda_4 + H_{24}(\lambda), \]
where
\[ H_{24}(\lambda) = \lambda_5\sum_{a=0}^\infty \sum_{\substack{b=-a\\ b\neq 0}}^{-\infty} \frac{(-b-1)!^2}{a!^2 (-2a-2b+1)!}\bigg(\frac{\lambda_1\lambda_3}{\lambda_5^2}\bigg)^a \bigg(\frac{\lambda_2\lambda_4}{\lambda_5^2}\bigg)^b. \]

We can now apply Proposition 5.3 to get solutions of the $A$-hypergeometric system.  If we take $l=(-1,0,-1,0,2)$ and $l'=(0,1,0,1,-2)$, we get the solution
\begin{multline*}
\lambda_5\log(\lambda_5^2/\lambda_1\lambda_3)\log(\lambda_2\lambda_4/\lambda_5^2) + 2G_5(\lambda) \big(\log(\lambda_2\lambda_4/\lambda_5^2) - \log(\lambda_5^2/\lambda_1\lambda_3)\big) \\
-4H_{55}(\lambda) + 2\sum_{i=1}^4 H_{i5}(\lambda).
\end{multline*}
Note that this solution lies in the Nilsson ring determined by the set of exponents~$L_{v,\hat{5}}$.  But the solutions corresponding to the choices $l=l'=(-1,0,-1,0,2)$ and $l=l'=(0,1,0,1,-2)$ do not lie in any Nilsson ring.  For example, if $l=l'=(-1,0,-1,0,2)$, then the solution involves $G_5$, $H_{15}$, $H_{35}$, and $H_{55}$ (with support~$L_{v,\hat{5}}$) and $H_{13}$ (with support $L_{v,\hat{1}\hat{3}}$).  The total support of the solution is $L_{v,\hat{5}}\cup L_{v,\hat{1}\hat{3}}$, which does not lie in any pointed (i.e., strictly convex) cone.
}
\end{example}

\section{Families of complete intersections}

One can associate an $A$-hypergeometric system to a family of complete intersections in the torus.  For a certain choice of $\beta$, one obtains a system of interest for applications to mirror symmetry.  We describe this system, show that it has a complete set of quasisolutions, and make a conjecture regarding the integrality of the associated mirror maps.  

Consider sets $A_i = \{{\bf a}_0^{(i)},{\bf a}_1^{(i)},\dots,{\bf a}_{N_i}^{(i)}\}\subseteq{\mathbb Z}^n$ for $i=1,\dots,M$.  The elements ${\bf a}_0^{(i)}$ will play a special role in what follows.  Our system will be related to the family of complete intersections in the $n$-torus ${\mathbb T}^n$ over ${\mathbb C}$ defined by the equations
\[ f_{i,\lambda}(x) = \sum_{j=0}^{N_i} \lambda_j^{(i)}x^{{\bf a}_j^{(i)}} = 0\quad\text{for $i=1,\dots,M$.} \]
Define $\hat{\bf a}_j^{(i)}\in{\mathbb Z}^{n+M}$ by
\[\hat{\bf a}_j^{(i)} = ({\bf a}_j^{(i)},0,\dots,0,1,0,\dots,0), \]
where the `1' occurs in the $(n+i)$-th entry.  We consider the $A$-hypergeometric system associated to the set
\[ A = \{ \hat{\bf a}_j^{(i)}\mid i=1,\dots,M,\; j=0,\dots,N_i\}\subseteq{\mathbb Z}^{n+M}. \]
We take $\beta = -\sum_{i=1}^M \hat{\bf a}_0^{(i)}$.  If we define $v=(v_j^{(i)})\in{\mathbb C}^{\sum_{i=1}^M(N_i+1)}$ with
\begin{equation}
v_j^{(i)} = \begin{cases} -1 & \text{if $j=0$,} \\  0 & \text{if $j\neq 0$,} \end{cases}
\end{equation}
then $\sum_{i=1}^M\sum_{j=0}^{N_i} v_j^{(i)}\hat{\bf a}_j^{(i)} = \beta$.  By \cite[Proposition~5.11]{AS}, the vector $v$ has minimal negative support and the corresponding series solution $F(\lambda)$ of the $A$-hypergeometric system with parameter $\beta$ has integer coefficients.  

Note that
\[ L_v = \{l=(l^{(i)}_j)\in L\mid \text{$l^{(i)}_0\leq 0$ for all $i$ and $l_j^{(i)}\geq 0$ for all $j\neq 0$ and all $i$}\}. \]
Since two indices ($i$ and $j$) are needed to describe elements of $L$, we need to modify our earlier notation.  For a vector $z=(z_m^{(k)})\in{\mathbb C}^{\sum_{i=1}^M (N_i+1)}$ we refer to the set
\[ \widehat{((i),j)}\text{-}{\rm nsupp}(z) = \{ ((k),m)\neq ((i),j)\mid \text{$z_m^{(k)}$ is a negative integer}\} \]
as the $\widehat{((i),j)}$-negative support of $z$.  And we denote by $L_{v,\widehat{((i),j)}}$ the set
\[ L_{v,\widehat{((i),j)}} = \{ l\in L \mid \widehat{((i),j)}\text{-}{\rm nsupp}(v+l)=\widehat{((i),j)}\text{-}{\rm nsupp}(v)\}. \]

Define $\delta = \sum_{i=1}^M {\bf a}_0^{(i)}\in{\mathbb Z}^n$, so that
\[ \beta = (-\delta;-1,\dots,-1)\in{\mathbb Z}^{n+M}. \]
Let $\Delta_i\subseteq{\mathbb R}^n$, $i=1,\dots,M$, be the convex hull of the set $A_i$, and let $\Delta = \sum_{i=1}^M\Delta_i$ be their Minkowski sum.  The main result of this section is the following proposition.
\begin{proposition}
Suppose that $\delta$ is the unique interior lattice point of $\Delta$.  Then $v$ has minimal $\widehat{((i),j)}$-negative support for all $i,j$, hence the $A$-hypergeometric system with parameter $\beta$ has a complete set of quasisolutions.  Furthermore, the total support $\bigcup_{i=1}^M \bigcup_{j=0}^{N_i} L_{v,\widehat{((i),j)}}$ lies in a pointed cone, so the quasisolutions all lie in a common Nilsson ring.
\end{proposition}

Let $F(\lambda)\log\lambda_j^{(i)} + G^{(i)}_j(\lambda)$ be the quasisolution of Proposition 6.2 corresponding to the variable $\lambda_j^{(i)}$.  Since this quasisolution lies in a Nilsson ring and the coefficient of the term of the series $F(\lambda)$ (resp.\ $G^{(i)}_j(\lambda)$) corresponding to ${\bf 0}\in L$ is $1$ (resp.\ $0$), the series
\[ q_j^{(i)}(\lambda) = \lambda_j^{(i)}\exp\big(G^{(i)}_j(\lambda)/F(\lambda)\big) \]
is well defined and has support in the pointed cone of Proposition~6.2.
\begin{conjecture}
If $\delta$ is the unique interior lattice point of $\Delta$, then the series $q^{(i)}_j(\lambda)$ has integer coefficients.
\end{conjecture}

We give some examples related to this conjecture after the proof of Proposition ~6.2, which will require several steps.

Let $\hat{\Delta}_i\subseteq{\mathbb R}^{n+M}$ be the convex hull of the set $\hat{A}_i = \{\hat{\bf a}_0^{(i)}, \dots,\hat{\bf a}_{N_i}^{(i)}\}$ and let $\hat{\Delta}\subseteq{\mathbb R}^{n+M}$ be the convex hull of the set $A=\bigcup_{i=1}^M \hat{A}_i$.  Let $C(\hat{\Delta})\subseteq{\mathbb R}^{n+M}$ be the real cone generated by $\hat{\Delta}$.  Let $x_1,\dots,x_n,y_1,\dots,y_M$ be the coordinate functions on ${\mathbb R}^{n+M}$.  We shall be interested in the set $\Gamma$ defined by
\[ \Gamma = C(\hat{\Delta})\cap\{(x;y)\in{\mathbb R}^{n+M}\mid y_1+\cdots+y_M = M\}. \]
Note that $\Gamma$ is the convex hull of $\bigcup_{i=1}^M M\hat{\Delta}_i$, where $M\hat{\Delta}_i$ denotes the dilation of~$\hat{\Delta}_i$ by the factor $M$.  In particular, a point $(\epsilon;1,\dots,1)$ lies in $\Gamma$ if and only if $\epsilon$ lies in~$\Delta$.  Set
\[ \gamma = \sum_{i=1}^M \hat{\bf a}_0^{(i)} = (\delta;1,\dots,1)\in\Gamma. \]
\begin{lemma}
The point $\delta$ is the unique interior lattice point of $\Delta$ if and only if $\gamma$ is the unique interior lattice point of $\Gamma$.
\end{lemma}

\begin{proof}
The result is clear in the case $M=1$ since one has then $x\in\Delta$ if and only if $(x;1)\in\Gamma$.  So suppose that $M>1$.  In that case the hyperplanes $y_i=0$ are hyperplanes of support of $C(\hat{\Delta})$, hence the interior lattice points of $\Gamma$ must be of the form $(\epsilon;1,\dots,1)$, where $\epsilon$ is a lattice point of $\Delta$.

Suppose  that $\epsilon$ is a boundary point of $\Delta$.  Then there exists a linear form $h$ on~${\mathbb R}^n$, not constant on $\Delta$, such that
\begin{equation}
h(x)\leq h(\epsilon)\quad\text{for all $x\in\Delta$.} 
\end{equation}
Write $\epsilon = \sum_{i=1}^M {\bf b}_i$, where ${\bf b}_i\in\Delta_i$ for $i=1,\dots,M$.  We claim that $h$ assumes its maximum value on $\Delta_i$ at the point ${\bf b}_i$.  To see this, suppose that $h$ assumes its maximum on $\Delta_i$ at a point ${\bf b}'_i$ and take $x=\sum_{i=1}^M {\bf b}'_i$ in (6.5).  This gives
\[ \sum_{i=1}^M h({\bf b}'_i)\leq\sum_{i=1}^M h({\bf b}_i), \]
so $h({\bf b}_i) = h({\bf b}'_i)$ for all $i$.  Let $h'$ be the linear form on ${\mathbb R}^{n+M}$ defined by $h'(x;y) = h(x)-\sum_{i=1}^M h({\bf b}_i)y_i$.  Let $(x';y')\in M\hat{\Delta}_i$.  Then $(x';y') = M\cdot(x;y)$ with $(x;y)\in\hat{\Delta}_i$.  We thus have
\[ h'(x';y') = M\cdot h'(x;y) = h(x)-h({\bf b}_i)\leq 0. \]
Since $h'$ is nonpositive on each $M\hat{\Delta}_i$, it is nonpositive on $\Gamma$, and it is nonconstant on $\Gamma$ since $h$ is nonconstant on $\Delta$.  Furthermore,
\[ h'(\epsilon;1,\dots,1) = h(\epsilon)-\sum_{i=1}^M h({\bf b}_i), \]
which shows that $(\epsilon;1,\dots,1)$ is a boundary point of $\Gamma$.

Conversely, let $\epsilon\in\Delta$ and suppose that $(\epsilon;1,\dots,1)$ is a boundary point of~$\Gamma$.  Then there exists a linear form $h'(x;y)$ on ${\mathbb R}^{n+M}$ such that (i)~$h'(\epsilon;1,\dots,1) = 0$, (ii)~$h'$ is nonpositive on all $M\hat{\Delta}_i$, and (iii)~$h'$ assumes a negative value on some~$M\hat{\Delta}_i$.  Write
\[ h'(x;y) = h(x) - \sum_{i=1}^M b_iy_i, \]
where $h$ is a linear form on ${\mathbb R}^n$.  Then (i)~$h(\epsilon) = \sum_{i=1}^M b_i$, (ii)~$h(z_i)\leq b_i$ for all $z_i\in\Delta_i$ and (iii)~there exists $i_0$ and $z\in\Delta_{i_0}$ such that $h(z)<b_{i_0}$.  This implies that $h(x)\leq\sum_{i=1}^M b_i$ for all $x\in\Delta$ and that $h$ assumes a value $<\sum_{i=1}^M b_i$ on $\Delta$.  It follows that $\epsilon$ is a boundary point of $\Delta$.  
\end{proof}

\begin{proof}[Proof of Proposition $6.2$]
We first show that $v$ has minimal $\widehat{((i),j)}$-negative support for all $i,j$ by carrying out the argument for two representative cases, namely, $i=1$, $j=0$ and $i=1$, $j=1$.  We first observe that for all $l=(l^{(i)}_j)\in L$ we have the relation
\begin{equation}
\sum_{i=1}^M \sum_{j=0}^{N_i} l_j^{(i)}\hat{\bf a}_j^{(i)} = {\bf 0}.
\end{equation}
Furthermore, for $1\leq I\leq M$, the $(n+I)$-th coordinate of $\hat{\bf a}_j^{(i)}$ equals 1 if $i=I$ and equals 0 if $i\neq I$, so we also have the relation
\begin{equation}
\sum_{j=0}^{N_I} l^{(I)}_j = 0.
\end{equation}

Take $i=1$ and $j=0$.  To show that $v$ has minimal $\widehat{((1),0)}$-negative support, we need to show there is no $l=(l_j^{(i)})\in L$ such that $l_j^{(i)}\geq 0$ for $i=1,\dots,M$ and $j=1,\dots,N_i$ and such that $l_0^{(I)}\geq 1$ for some $I\neq 1$.  Equation (6.7) implies that there is no such $l$, hence $v$ has minimal $\widehat{((1),0)}$-negative support.

Now take $i=1$ and $j=1$.  To show that $v$ has minimal $\widehat{((1),1)}$-negative support, we need to show there is no $l=(l_j^{(i)})\in L$ such that $l_j^{(i)}\geq 0$ for $i=1,\dots,M$ and $j=1,\dots,N_i$ except for $i=j=1$ and such that $l_0^{(I)}\geq 1$ for some $I$.  Suppose such an $l$ existed.  We prove first that $I=1$.  Since we are assuming $l_j^{(I)}\geq 0$ for all $j\neq 0$ if $I\geq 2$, Equation (6.7) implies that we cannot have $l_0^{(I)}\geq 1$ for $I\geq 2$.  It follows that if such an $l$ existed, it would have the properties $l_0^{(1)}\geq 1$, $l_1^{(1)}<0$, $l_0^{(i)}\leq 0$ for~$i\geq 2$, and $l_j^{(i)}\geq 0$ for all other $i,j$.  

Rearranging (6.6) gives the relation
\begin{equation}
-l_1^{(1)}\hat{\bf a}_1^{(1)} - \sum_{i=2}^M l_0^{(i)}\hat{\bf a}_0^{(i)} = l_0^{(1)}\hat{\bf a}_0^{(1)} + \sum_{j=2}^{N_1} l_j^{(1)}\hat{\bf a}_j^{(1)} + \sum_{i=2}^M\sum_{j=1}^{N_i} l_j^{(i)}\hat{\bf a}_j^{(i)}.
\end{equation}
Since the coefficients of the elements of $A$ in this equation are nonnegative integers, both sides represent a lattice point in $C(\hat{\Delta})$.  By Lemma 6.4, $\gamma = \sum_{i=1}^M \hat{\bf a}_0^{(i)}$ is the unique interior lattice point of $\Gamma$, so the sum $\hat{\bf a}_1^{(1)} + \sum_{i=2}^M \hat{\bf a}_0^{(i)}$ must lie on some codimension-one face of $\Gamma$.  But if a nonnegative linear combination of vectors in a cone lies on a face of that cone, then each of those vectors must lie on that face.  It follows that the vectors $\hat{\bf a}_1^{(1)}$ and $\{\hat{\bf a}_0^{(i)}\}_{i=2}^M$ all lie on the same codimension-one face of~$C(\hat{\Delta})$, hence the left-hand side of (6.8) lies on that face.  The same reasoning applied to the right-hand side of (6.8) then shows that $\hat{\bf a}_0^{(1)}$ lies on that face also.  But this implies that $\gamma$ lies on a codimension-one face of $C(\hat{\Delta})$, contradicting Lemma~6.4.  Thus $v$ has minimal $\widehat{((1),1)}$-negative support.

As a first step towards proving the second assertion of Proposition 6.2, we show that each $L_{v,\widehat{((i),j)}}$ lies in a pointed cone.  If $l=(l_m^{(k)})\in L_{v,\widehat{((i),0)}}$, then $l_m^{(i)}\geq 0$ for all $m\geq 1$, hence (6.7) implies that $l_0^{(i)}\leq 0$.  So for $i=1,\dots,M$ we have
\begin{equation}
L_{v,\widehat{((i),0)}} = L_v = \{l=(l_m^{(k)})\in L\mid \text{$l_0^{(k)}\leq 0$, $l_m^{(k)}\geq 0$ for all $k$ and all $m\geq 1$}\}.
\end{equation}
For $i=1,\dots,M$ and $j=1,\dots,N_i$ we have from the definition
\begin{equation}
L_{v,\widehat{((i),j)}} =  \{l=(l_m^{(k)})\in L\mid \text{$l_0^{(k)}\leq 0$, $l_m^{(k)}\geq 0$ for all $(k,m)\neq (i,j)$}\}.
\end{equation}

To say that $L_{v,\widehat{((i),j)}}$ lies in a pointed cone is equivalent to saying that any expression of ${\bf 0}$ as a linear combination with nonnegative integer coefficients of elements of $L_{v,\widehat{((i),j)}}$ is trivial.  Fix $i,j$, let $\{\xi(h) = (\xi(h)_m^{(k)})\}_{h=1}^H$ be a subset of~$L_{v,\widehat{((i),j)}}$ and let $\{b(h)\}_{h=1}^H$ be positive integers.  We need to show that if 
\begin{equation}
\sum_{h=1}^H b(h)\xi(h) = {\bf 0},
\end{equation}
then $\xi(h) = {\bf 0}$ for all $h$.  By (6.9) and (6.10) we have $\xi(h)_0^{(k)}\leq 0$ for all $h,k$, so~(6.11) implies that in fact $\xi(h)_0^{(k)}=0$ for all $h,k$.  If $j=0$ we have by (6.9) that $\xi(h)_m^{(k)}\geq 0$ for $m\geq 1$, so again (6.11) implies that $\xi(h)_m^{(k)}=0$ for all $h,k$ and all~$m\geq 1$; we conclude that $\xi(h)={\bf 0}$ for all $h$ when $j=0$.  

Now suppose that $j\geq 1$.  Then for $m\geq 1$ and $(k,m)\neq(i,j)$ we have by~(6.10) that $\xi(h)_m^{(k)}\geq 0$, so (6.11) implies that $\xi(h)_m^{(k)}=0$ for all $h$, $(k,m)\neq(i,j)$, and $m\geq 1$.  Since $\xi(h)\in L$, Equation (6.7) implies
\[ \sum_{m=0}^{N_i} \xi(h)_m^{(i)} = 0. \]
We have already shown that $\xi(h)_m^{(i)}=0$ for all $m\neq j$, so this equation shows that $\xi(h)_j^{(i)}=0$ also.  We conclude that $\xi(h) = {\bf 0}$ for all $h$ when $j\geq 1$, thus each $L_{v,\widehat{((i),j)}}$ lies in a pointed cone.  

We now proceed to show that $\bigcup_{i,j} L_{v,\widehat{((i),j)}}$ lies in a pointed cone.  We need to show that the only linear combination of elements of $\bigcup_{i,j} L_{v,\widehat{((i),j)}}$ with nonnegative integer coefficients that equals zero is the trivial one.  Since the sets $L_{v,\widehat{((i),j)}}$ are closed under taking linear combinations with nonnegative integer coefficients, we can group the terms of the linear combination coming from the same $L_{v,\widehat{((i),j)}}$ together into an element $\xi(i,j)\in L_{v,\widehat{((i),j)}}$, giving the equation
\begin{equation}
\sum_{i,j} \xi(i,j) = {\bf 0}. 
\end{equation}
It suffices to show that $\xi(i,j)={\bf 0}$ for all $i,j$:  since we have just proved that $L_{v,\widehat{((i),j)}}$ lies in a pointed cone, the vanishing of $\xi(i,j)$ implies that the expression of $\xi(i,j)$ as a nonnegative linear combination of elements of $L_{v,\widehat{((i),j)}}$ is trivial.

\begin{lemma}  
If $\xi(i,j)\neq{\bf 0}$, then
\begin{align}
\xi(i,j)^{(k)}_m &= 0\quad\text{if $k\neq i$,} \\
\xi(i,j)^{(i)}_0 &= 0, \\
\xi(i,j)^{(i)}_m &\geq 0\quad\text{if $m\neq j$,} \\
\xi(i,j)^{(i)}_j &< 0.
\end{align}
In particular, we must have $j\geq 1$.
\end{lemma}

\begin{proof}[Proof of Lemma $6.13$]
We have from (6.9) and (6.10) that $\xi(i,j)_0^{(k)}\leq 0$ for all $i,j,k$, so~(6.12) implies that
\begin{equation}
\xi(i,j)^{(k)}_0 = 0\quad\text{for all $k$.}
\end{equation}
By (6.9) and (6.10) we have
\begin{equation}
\xi(i,j)^{(I)}_m\geq 0\quad\text{for $I\neq i$ and $m\geq 1$.}
\end{equation}
Since $\xi(i,j)\in L$, Equation (6.7) becomes
\begin{equation}
\sum_{m=0}^{N_I} \xi(i,j)^{(I)}_m = 0,
\end{equation}
so by (6.18) and (6.19) we get $\xi(i,j)^{(I)}_m = 0$ for all $I\neq i$ and all $m$.  This establishes~(6.14).  Equation (6.15) follows from (6.18).

Equations (6.9) and (6.10) imply that $\xi(i,j)^{(i)}_m\geq 0$ if both $m\geq 1$ and~$m\neq j$, which establishes (6.16).  Taking $I=i$ in (6.20) and using (6.18) gives
\begin{equation}
\sum_{m=1}^{N_i}\xi(i,j)^{(i)}_m = 0.
\end{equation}
If $\xi(i,j)_j^{(i)}\geq 0$, then (6.16) and (6.21) imply $\xi(i,j)_m^{(i)} = 0$ for all $m$.  Combined with (6.14) this says that $\xi(i,j)={\bf 0}$, contradicting our hypothesis.  We must therefore have $\xi(i,j)^{(i)}_j<0$, which establishes (6.17).  Equations (6.15) and (6.17) are inconsistent when $j=0$, so we must have $j\geq 1$.
\end{proof}

We draw the following conclusion from Lemma 6.13.
\begin{corollary}
If $\xi(i,j)\neq{\bf 0}$, then $\hat{\bf a}_j^{(i)}$ lies in the convex hull of the set
\[ \{ \hat{\bf a}_m^{(i)}\mid m=1,\dots,\hat{\jmath},\dots,N_i\}. \]
\end{corollary}

\begin{proof}[Proof of Corollary $6.22$]
Since $\xi(i,j)\in L$, we have the relation
\[ \sum_{k=1}^M \sum_{m=0}^{N_k} \xi(i,j)^{(k)}_m\hat{\bf a}_m^{(k)} = 0. \]
When $\xi(i,j)\neq{\bf 0}$, Equations (6.14) and (6.15) show that this simplifies to
\begin{equation}
\sum_{m=1}^{N_i} \xi(i,j)^{(i)}_m\hat{\bf a}_m^{(i)} = {\bf 0}.
\end{equation}
The $(n+i)$-th coefficient of $\hat{\bf a}_m^{(i)}$ is $1$, so (6.23) gives
\begin{equation}
\sum_{m=1}^{N_i} \xi(i,j)^{(i)}_m = 0.
\end{equation}
By (6.17) we can solve (6.23) for $\hat{\bf a}^{(i)}_j$:
\begin{equation} 
\hat{\bf a}^{(i)}_j = \sum_{\substack{m=1\\ m\neq j}}^{N_i} \bigg(-\frac{\xi(i,j)^{(i)}_m}{\xi(i,j)^{(i)}_j}\bigg) \hat{\bf a}_m^{(i)}.
\end{equation}
By (6.16), (6.17), and (6.24), the coefficients on the right-hand side of (6.25) are nonnegative and sum to $1$, so (6.25) implies that $\hat{\bf a}_j^{(i)}$ lies in the convex hull of 
\[ \{ \hat{\bf a}_m^{(i)}\mid m=1,\dots,\hat{\jmath},\dots,N_i\}. \]
\end{proof}

Let $A'\subseteq A$ be defined by
\[ A' = \{ \hat{\bf a}^{(k)}_m\mid \text{$\xi(i,j)^{(k)}_m\neq 0$ for some $\xi(i,j)$} \}. \]
Choose $\hat{\bf a}_{m_0}^{(k_0)}\in A'$ to be a vertex of the convex hull of $A'$.  In particular, this implies that $\hat{\bf a}_{m_0}^{(k_0)}$ does not lie in the convex hull of any subset of $A'$ not containing~$\hat{\bf a}_{m_0}^{(k_0)}$.  By Corollary~6.22 we then have $\xi(k_0,m_0) = {\bf 0}$, so $(i,j)\neq(k_0,m_0)$ for any pair $(i,j)$ such that $\xi(i,j)_{m_0}^{(k_0)}\neq 0$.  Lemma 6.13 now implies that $\xi(i,j)_{m_0}^{(k_0)}\geq 0$ for all
such~$i,j$.  Furthermore, $\xi(i,j)^{(k_0)}_{m_0}>0$ for some $i,j$ because $\hat{\bf a}_{m_0}^{(k_0)}\in A'$.  This implies that the $((k_0),m_0)$-coefficient on the right-hand side of (6.12) is $>0$, a contradiction, so there cannot be any $i,j$ such that $\xi(i,j)\neq{\bf 0}$.
\end{proof}

\begin{example}{\rm
Let $M=1$, $N_1=6$, and let $A=\{\hat{\bf a}_0^{(1)},\dots,\hat{\bf a}_6^{(1)}\}\subseteq{\mathbb Z}^5$ be the columns of the matrix
\[ \left( \begin{array}{rrrrrrr} 0 & 1 & 0 & -1 & 0 & 0 & 0 \\ 0 & 0 & 1 & -1 & 0 & 0 & 0 \\ 0 & 0 & 0 & 0 & 1 & 0 & -1 \\
0 & 0 & 0 & 0 & 0 & 1 & -1 \\ 1 & 1 & 1 & 1 & 1 & 1 & 1 \end{array} \right). \]
We take $\beta = (0,\dots,0,-1)\in{\mathbb Z}^5$ and $v=(-1,0,\dots,0)\in{\mathbb C}^7$.  Then ${\bf a}_0^{(1)} = (0,0,0,0)\in{\mathbb Z}^4$ is the unique interior lattice point in the convex hull of ${\bf a}_1^{(1)},\dots,{\bf a}_6^{(1)}$.  We have
\[ L = \{(-3l-3m,l,l,l,m,m,m)\in{\mathbb Z}^7\mid l,m\in{\mathbb Z}\} \]
and 
\[ L_v = \{(-3l-3m,l,l,l,m,m,m)\in{\mathbb Z}^7\mid l,m\in{\mathbb Z}_{\geq 0}\}. \]
We thus get the solution
\[ F(\lambda) = \lambda_0^{-1} \sum_{l,m=0}^\infty \frac{(3l+3m)!}{l!^3m!^3} \bigg(-\frac{\lambda_1\lambda_2\lambda_3}{\lambda_0^3}\bigg)^l\bigg(-\frac{\lambda_4\lambda_5\lambda_6}{\lambda_0^3}\bigg)^m. \]
By Proposition 6.2 the vector $v$ has minimal $\widehat{((1),j)}$-negative support for $j=0,\dots,6$ and we clearly have
\[ L_{v,\widehat{((1),0)}} = \cdots = L_{v,\widehat{((1),6)}} = L_v. \]
Theorem 4.11 gives us the quasisolutions $F(\lambda)\log\lambda_j^{(1)} + G_j^{(1)}(\lambda)$ for $j=0,\dots,6$, where
\[ G_0^{(1)}(\lambda) = -\lambda_0^{-1} \sum_{\substack{l,m=0\\ (l,m)\neq(0,0)}}^\infty \frac{(3l+3m)!}{l!^3m!^3} \bigg(\sum_{i=1}^{3l+3m} \frac{1}{i}\bigg)\bigg(-\frac{\lambda_1\lambda_2\lambda_3}{\lambda_0^3}\bigg)^l\bigg(-\frac{\lambda_4\lambda_5\lambda_6}{\lambda_0^3}\bigg)^m, \]
\[ G_j^{(1)}(\lambda) = -\lambda_0^{-1} \sum_{l=1}^\infty\sum_{m=0}^\infty \frac{(3l+3m)!}{l!^3m!^3} \bigg(\sum_{i=1}^{l} \frac{1}{i}\bigg)\bigg(-\frac{\lambda_1\lambda_2\lambda_3}{\lambda_0^3}\bigg)^l\bigg(-\frac{\lambda_4\lambda_5\lambda_6}{\lambda_0^3}\bigg)^m \]
for $j=1,2,3$, and 
\[ G_j^{(1)}(\lambda) = -\lambda_0^{-1} \sum_{l=0}^\infty\sum_{m=1}^\infty \frac{(3l+3m)!}{l!^3m!^3} \bigg(\sum_{i=1}^{m} \frac{1}{i}\bigg)\bigg(-\frac{\lambda_1\lambda_2\lambda_3}{\lambda_0^3}\bigg)^l\bigg(-\frac{\lambda_4\lambda_5\lambda_6}{\lambda_0^3}\bigg)^m \]
for $j=4,5,6$.  The integrality of the series $\exp\big(G_j^{(1)}(\lambda)/F(\lambda)\big)$ for $j=0,\dots,6$ asserted by Conjecture 6.3 is due in this case to Krattenthaler-Rivoal\cite{KR}.
}
\end{example}

\begin{example} {\rm
Let $M=1$, $N_1=4$, and let $A=\{\hat{\bf a}_0^{(1)},\dots,\hat{\bf a}_4^{(1)}\}\subseteq{\mathbb Z}^3$ be the columns of the matrix
\[ \left( \begin{array}{rrrrr} 0 & 1 & -1 & 1 & 1 \\ 0 & 1 & 0 & -1 & 0 \\ 1 & 1 & 1 & 1 & 1 \end{array} \right). \]
Take $\beta = (0,0,-1)\in{\mathbb Z}^3$ and $v=(-1,0,0,0,0)\in{\mathbb C}^5$.  Then ${\bf a}_0^{(1)}=(0,0)\in{\mathbb Z}^2$ is the unique interior lattice point in the convex hull of ${\bf a}_1^{(1)},\dots,{\bf a}_4^{(1)}$.  We have
\[ L = \{(-4l-2m,l,2l+m,l,m)\in{\mathbb Z}^5\mid l,m\in{\mathbb Z}\} \]
\and 
\[ L_v = \{(-4l-2m,l,2l+m,l,m)\in{\mathbb Z}^5\mid l,m\in{\mathbb Z}_{\geq 0} \}. \]
This gives the solution
\[ F(\lambda) = \lambda_0^{-1} \sum_{l=0}^\infty\sum_{m=0}^\infty \frac{(4l+2m)!}{l!^2 m! (2l+m)!} \bigg(
\frac{\lambda_1\lambda_2^2\lambda_3}{\lambda_0^4}\bigg)^l \bigg(\frac{\lambda_2\lambda_4}{\lambda_0^2}\bigg)^m. \]
By Proposition 6.2 the vector $v$ has minimal $\widehat{((1),j)}$-negative support for $j=0,\dots,4$.  We have
\[ L_{v,\widehat{((1),0)}} = \dots = L_{v,\widehat{((1),3)}} = L_v \]
but 
\[ L_{v,\widehat{((1),4)}} = \{(-4l-2m,l,2l+m,l,m)\in{\mathbb Z}^5\mid l,2l+m\in{\mathbb Z}_{\geq 0} \}. \]
Theorem 4.11 gives the quasisolutions $F(\lambda)\log\lambda_j^{(1)} + G_j^{(1)}(\lambda)$ for $j=0,\dots,4$, where
\[ G_0^{(1)}(\lambda) = -\lambda_0^{-1}\sum_{\substack{l,m = 0\\ (l,m)\neq (0,0)}}^\infty \frac{(4l+2m)!}{l!^2 m! (2l+m)!} \bigg(\sum_{i=1}^{4l+2m} \frac{1}{i}\bigg)\bigg(\frac{\lambda_1\lambda_2^2\lambda_3}{\lambda_0^4}\bigg)^l \bigg(\frac{\lambda_2\lambda_4}{\lambda_0^2}\bigg)^m, \]
\[ G_1^{(1)}(\lambda) = G_3^{(1)}(\lambda) =  -\lambda_0^{-1}\sum_{l=1}^\infty\sum_{m=0}^\infty  \frac{(4l+2m)!}{l!^2 m! (2l+m)!} \bigg(\sum_{i=1}^l \frac{1}{i}\bigg)\bigg(\frac{\lambda_1\lambda_2^2\lambda_3}{\lambda_0^4}\bigg)^l \bigg(\frac{\lambda_2\lambda_4}{\lambda_0^2}\bigg)^m, \]
\[ G_2^{(1)}(\lambda) =  -\lambda_0^{-1}\sum_{\substack{l,m = 0\\ (l,m)\neq (0,0)}}^\infty \frac{(4l+2m)!}{l!^2 m! (2l+m)!} \bigg(\sum_{i=1}^{2l+m} \frac{1}{i}\bigg)\bigg(\frac{\lambda_1\lambda_2^2\lambda_3}{\lambda_0^4}\bigg)^l \bigg(\frac{\lambda_2\lambda_4}{\lambda_0^2}\bigg)^m, \]
\begin{multline*}
G_4^{(1)}(\lambda) =  -\lambda_0^{-1}\sum_{l=0}^\infty\sum_{m=1}^\infty  \frac{(4l+2m)!}{l!^2 m! (2l+m)!} \bigg(\sum_{i=1}^{m} \frac{1}{i}\bigg)\bigg(\frac{\lambda_1\lambda_2^2\lambda_3}{\lambda_0^4}\bigg)^l \bigg(\frac{\lambda_2\lambda_4}{\lambda_0^2}\bigg)^m \\ 
-\lambda_0^{-1} \sum_{l=0}^\infty \sum_{m=-2l}^{-1} \frac{(4l+2m)!(-1)^{-m}(-m-1)!}{l!^2  (2l+m)!} \bigg(\frac{\lambda_1\lambda_2^2\lambda_3}{\lambda_0^4}\bigg)^l \bigg(\frac{\lambda_2\lambda_4}{\lambda_0^2}\bigg)^m. 
\end{multline*}
The integrality of the series $\exp\big(G^{(1)}_j(\lambda)/F(\lambda)\big)$ for $j=0,\dots,3$ asserted by Conjecture~6.3 may follow from the work of Delaygue\cite{D}, however, the integrality for $j=4$ seems to be an open question.
}
\end{example}

\end{document}